\documentclass[ejs]{imsart}
\arxiv{arXiv:1507.07412}
\RequirePackage{amsthm,amsmath,natbib}
\RequirePackage[colorlinks,citecolor=blue,urlcolor=blue]{hyperref}

\usepackage{amsfonts}
\usepackage{amssymb} 
\usepackage{graphicx}
\usepackage{amssymb}
\usepackage{amsthm}
\usepackage[utf8]{inputenc} % set input encoding (not needed with XeLaTeX)
\usepackage{graphicx} % support the \includegraphics command and options
\usepackage{booktabs} % for much better looking tables
\usepackage{array} % for better arrays (eg matrices) in maths
\usepackage{paralist} % very flexible & customisable lists (eg. itemize/itemize, etc.)
\usepackage{verbatim} 
\usepackage{subfig} % make it possible to include more than one captioned figure/table in a single float
\usepackage{url}
\usepackage{mathrsfs}

\startlocaldefs
\numberwithin{equation}{section}
\long\def\beginskip#1\endskip{}
\newcommand{\E}{\mathord{\rm E}}
\newtheorem{theorem}{Theorem}
\newtheorem{lemma}{Lemma}

\newtheorem{proposition}{Proposition}
\newcommand{\norm}[1]{\|{#1}\|}

\newcommand{\bbE}{\mathbb{E}}
\newcommand{\bbP}{\mathbb{P}}
\DeclareMathOperator{\diam}{diam}
\newcommand{\calG}{\mathcal{P}}

\DeclareMathOperator{\DP}{DP}

\newcommand{\sgn}{\mathop{\rm sgn}\nolimits}

\newcommand{\RR}{\mathbb{R}}
\endlocaldefs

\begin{document}

\begin{frontmatter}

\title{Posterior contraction rates for deconvolution of Dirichlet-Laplace mixtures}
\runtitle{Contraction Rates for Dirichlet-Laplace mixtures}

\begin{aug}
\author{\fnms{Fengnan} \snm{Gao}\thanksref{t1}\ead[label=e1]{gaof@math.leidenuniv.nl}}
\and 
\author{\fnms{Aad} \snm{van der Vaart}\ead[label=e2]{avdvaart@math.leidenuniv.nl}\thanksref{t2}}
\thankstext{t1}{Research supported by the Netherlands Organization for Scientific Research.}
\thankstext{t2}{The research leading to these results has received funding from the European
  Research Council under ERC Grant Agreement 320637.}

\affiliation{Leiden University}

\address{Mathematical Institute \\
Leiden University \\
 Niels Bohrweg 1 \\
2333 CA Leiden, Netherlands. \\
\phantom{Email: \ }\printead*{e1}
\phantom{Email: \ }\printead*{e2}}  

\end{aug}

\begin{abstract}
We study nonparametric
Bayesian inference with location mixtures of the Laplace density and a Dirichlet process
prior on the mixing distribution. We derive a contraction rate of the corresponding posterior
distribution, both for the mixing distribution relative to the Wasserstein metric and for the
mixed density relative to the Hellinger and $L_q$ metrics.
\end{abstract}



\end{frontmatter}

\section{Introduction}
Consider statistical inference using the following hierarchical Bayesian model
for observations $X_1,\ldots, X_n$:
\begin{enumerate}
\item[(i)] A probability distribution $G$ on $\RR$ 
 is generated from the Dirichlet process prior $\DP(\alpha)$ with base measure $\alpha$.
\item[(ii)] An i.i.d.\ sample $Z_1,\ldots, Z_n$ is generated from $G$.
\item[(iii)] An i.i.d.\ sample $e_1,\ldots, e_n$ is generated from a known density $f$,
 independent of the other samples.
\item[(iv)] The observations are $X_i=Z_i+e_i$, for $i=1,\ldots, n$.
\end{enumerate}
In this setting the conditional density of the data $X_1,\ldots, X_n$ given $G$ is a sample from
the convolution $$p_G = f \ast G$$ of the density $f$ and the measure $G$. 
The scheme defines a conditional distribution of $G$ given the data $X_1,\ldots, X_n$, 
the \emph{posterior distribution} of $G$,
and consequently also posterior distributions for quantities that derive from $G$, including
the convolution density $p_G$. 
We are interested in whether this posterior distribution
can recover a ``true" mixing distribution $G_0$ if the observations $X_1,\ldots, X_n$
are in reality a sample from the mixed distribution $p_{G_0}$, for some given
probability distribution $G_0$. 

The main contribution of this paper is for the
case that $f$ is the Laplace density $f(x)=e^{-|x|}/2$. For distributions on
the full line Laplace mixtures seem the second
most popular class next to mixtures of the normal distribution, with applications in for instance
speech recognition or astronomy (\cite{Kotzetal}) and clustering problem in genetics (\cite{bailey1994fitting}).
For the present theoretical investigation the Laplace kernel 
is  interesting as a test case of a non-supersmooth kernel.

We consider two notions of recovery. 
The first notion measures the distance between the posterior of $G$ and $G_0$
through the \emph{Wasserstein metric} 
$$W_{k} (G, G') = \inf_{\gamma \in \Gamma (G, G')} 
\Bigl(\int |x- y|^k \, d\gamma (x, y)\Bigr)^{1/k},$$
where $\Gamma(G, G')$ is the collection of all \emph{couplings} $\gamma$
of $G$ and $G'$ into a bivariate measure with marginals $G$ and $G'$
(i.e. if $(x,y)\sim \gamma$, then $x \sim G $ and $y \sim G'$), and $k\ge 1$.
The Wasserstein metric is a classical metric on probability distributions, which is well
suited for use in obtaining rates of estimation of measures. It
is weaker than the total variation distance (which is more natural as a distance on densities), 
can be interpreted through transportation of
measure (see \cite{villani_optimal_2009}), and has also been used in applications such as as comparing the 
color histograms of digital images. 
Recovery of the posterior distribution relative to the Wasserstein metric was
considered by \cite{nguyen_convergence_2013}, within a general
mixing framework. We refer to this paper for further motivation of the Wasserstein metric for mixtures,
and to \cite{villani_optimal_2009} for general background on the Wasserstein metric.
In the present paper we improve the upper bound 
on posterior contraction rates given in \cite{nguyen_convergence_2013}, 
at least in the case of the Laplace mixtures, obtaining a rate of nearly $n^{-1/8}$ for $W_1$
(and slower rates for $k>1$).
Apparently the minimax rate of contraction for Laplace mixtures relative to the Wasserstein metric
is currently unknown. Recent work on recovery of a mixing distribution 
by non-Bayesian methods is given in \cite{zhang_cunhui_1990}.  It is not clear from our result
whether the upper bound $n^{-1/8}$ is sharp.

The second notion of recovery measures the distance of the posterior of $G$ to $G_0$
indirectly through the Hellinger or $L_q$-distances between the mixed 
densities $p_G$ and $p_{G_0}$. This is equivalent to studying the estimation of the
true density $p_{G_0}$ of the observations through the density $p_G$ under the posterior distribution.
As the Laplace kernel $f$ has Fourier transform 
\begin{equation*}
    \tilde f(\lambda) = \frac{1}{ 1 + \lambda^2},
\end{equation*}
it follows that the mixed densities $p_G$ have Fourier transforms satisfying
$$|\tilde p_G(\lambda)|\le \frac{1}{ 1 + \lambda^2}.$$
Estimation of a density with a polynomially decaying Fourier transform was first considered
in \cite{watson_estimation_1963}. According to their Theorem in Section~$3$A, a suitable kernel
estimator possesses a root mean square error of  $n^{-3/8}$ 
with respect to the $L_2$-norm  for estimating a density with Fourier transform that
decays exactly at the order 2. This rate is the ``usual rate''
$n^{-\alpha/(2\alpha+1)}$ of nonparametric estimation for smoothness $\alpha=3/2$.
This is understandable as  $|\tilde p(\lambda)|\lesssim 1/(1+|\lambda|^2)$
implies that $\int (1+|\lambda|^2)^{\alpha}|\tilde p(\lambda)|^2\,d\lambda<\infty$, 
for every $\alpha<3/2$, 
so that a density with Fourier transform decaying at square rate
belongs to any Sobolev class of regularity $\alpha<3/2$.
Indeed in \cite{golubev1992nonparametric},
the rate $n^{-\alpha/(2\alpha+1)}$ is shown to be minimax for estimating
a density in a Sobolev ball of functions on the line. In the present paper
we show that the posterior distribution of Laplace mixtures $p_G$
contracts to $p_{G_0}$ at the rate $n^{-3/8}$ up to a logarithm factor, 
relative to the $L_2$-norm and Hellinger distance, and also
establish rates for other $L_q$-metrics. Thus the Dirichlet posterior (nearly) attains
the minimax rate for estimating a density in a Sobolev ball of order $3/2$. 
It may be noted that the Laplace density itself
is H\"older of exactly order 1, which implies that Laplace mixtures are H\"older smooth of at least
the same order. This insight would suggest a rate $n^{-1/3}$ (the usual nonparametric 
rate for $\alpha=1$), which is slower than $n^{-3/8}$, and hence this insight is misleading.

Besides recovery relative to the Wasserstein metric and the induced metrics on $p_G$,
one might consider recovery relative to a metric on the distribution function on $G$.
Frequentist recovery rates for this problem were obtained in \cite{fan_optimal_1991} under some
restrictions. There is no simple
relation between these rates and rates for the other metrics. The same is true
for the rates for deconvolution of densities, as in \cite{fan_optimal_1991}. In fact, the Dirichlet prior
and posterior considered here are well known to concentrate on discrete distributions, and hence
are useless as priors for recovering a density of $G$. 

Contraction rates for Dirichlet mixtures of the normal kernel were considered in 
\cite{ghosal_entropies_2001, ghosal_posterior_2007, kruijer, shen_adaptive_2011,scricciolo}. 
The results in these papers are
driven by the smoothness of the Gaussian kernel, whence the same approach will fail
for the Laplace kernel. Nevertheless we borrow the idea of approximating the
true mixed density by a finite mixture, albeit that the approximation is
constructed in a different manner. Because more support points than in the Gaussian
case are needed to obtain a given quality of approximation,  higher entropy and lower prior mass
concentration result, leading to a slower rate of posterior contraction. To obtain the contraction
rate for the Wasserstein metrics we further derive a relationship of these metrics with a power
of the Hellinger distance, and next apply a variant of the 
contraction theorem in \cite{ghosal_convergence_2000}, which is included in the appendix of the paper.
Contraction rates of mixtures with other priors than the Dirichlet were considered
in \cite{scricciolo}.  Recovery of the mixing distribution is a deconvolution problem and as such 
can be considered an inverse problem. A general approach to posterior contraction rates in inverse
problems can be found in \cite{knapik2014general}, and results specific to deconvolution can be 
found in \cite{donnet2014posterior}. These authors are interested in deconvolving a (smooth) mixing density
rather than a mixing distribution, and hence their results are not directly comparable to the
results in the present paper.

The paper is organized as follows. In the next section we state the main results of the paper,
which are proved in the subsequent sections.
In Section~\ref{sec-approximation} we establish suitable finite approximations relative
to the $L_q$- and Hellinger distances. The $L_q$-approximations
also apply to other kernels than the Laplace, and are 
in terms of the tail decay of the kernel's characteristic function. 
In Sections~\ref{sec-entropy-mixture} and~\ref {sec-prior-mass} we apply these
approximations to obtain bounds on the entropy of the mixtures
relative to the $L_q$, Hellinger and Wasserstein metrics, and a lower
bound on the prior mass in a neighbourhood of the true density.
Sections~\ref{SectionProofsMain} and~\ref{SectionProofsMainTwo}  contain
 the proofs of the main results.

\subsection{Notation and preliminaries}
Throughout the paper integrals given
without limits are considered to be integrals over the real line $\mathbb{R}$.  
The $L_q$-norm is denoted
\[\norm{g}_q = \left(\int |g(x)|^q\, dx \right)^{1/q},\]
with $\norm{\cdot}_{\infty}$ being the uniform norm. The \emph{Hellinger distance}
on the space of densities  is given by
\[h(f,g) = \left(\int (f^{1/2}(x) - g^{1/2}(x))^2\, dx \right)^{1/2}.\]  
It is easy to see that $h^2(f, g) \le \norm{f - g}_1 \le 2 h(f,g)$,
for any two probability densities $f$ and $g$. Furthermore, 
if the densities $f$ and $g$ are uniformly bounded by a constant $M$, then
$\norm{f -g }_2 \le 2 \sqrt{M} h(f,g)$.
The Kullback-Leiber discrepancy and corresponding variance are denoted by
\[ K(p_0, p) = \int \log (p_0/p)\ dP_0, \qquad K_2(p_0, p) = \int (\log (p_0/p))^2\, dP_0 \]
with $P_0$ the measure corresponding to the density $p_0$.

We are primarily interested in the Laplace kernel, but a number of results are true for general
kernels $f$. The Fourier transform of a function $f$ and the inverse Fourier transform of a function
$\tilde f$ are given by 
\begin{equation*}
    \tilde{f}(\lambda) = \int e^{\imath \lambda x} f(x) dx, \qquad
    f(x) = \frac{1}{2 \pi} \int e^{-\imath \lambda x} \tilde{f}(\lambda) d\lambda.
\end{equation*}
For $\frac{1}{p} + \frac{1}{q} =1$ and $1 \le p\le 2$, \emph{Hausdorff-Young's inequality} gives that 
$\norm{f}_q \le (2\pi)^{-1/p} \norm{\tilde{f}}_p$.%    \label{eqn-fouier-bound-lp-lq}

The covering number  $N(\varepsilon, \Theta, \rho)$ of a metric space $(\Theta, \rho)$ is
the minimum number of $\varepsilon$-balls needed to cover the entire space $\Theta$.
\beginskip
  $D(\varepsilon, \Theta, \rho)$ is the packing number of $(\Theta, \rho)$, which is the maximum
  number of points that are pairwise at least $\varepsilon$-separated.  It is known that \(
  N(\varepsilon, \Theta, \rho) \le D(\varepsilon, \Theta, \rho) \le N(\varepsilon/2, \Theta, \rho)
  \).  
\endskip

Throughout the paper  $\lesssim$ denotes inequality up to a constant multiple, where the
constant is universal or fixed within the context.  Furthermore
$a_n \asymp b_n$ means  
$c \le\liminf_{n \rightarrow \infty} a_n/b_n \le \limsup_{n \rightarrow \infty} a_n/b_n \le C$,  for some
positive constants $c$ and $C$.  

%$X^{(n)}$ denotes the collection of data $\{ X_1, X_2, \dots, X_n \}$. 

We denote by $\mathcal{M}[-a,a]$ the set of all probability measures on a given
interval $[-a, a]$.

\section{Main Results}
Let $\Pi_n(\cdot | X_1,\ldots, X_n)$ be the posterior distribution  for $G$ in the scheme
(i)-(iv) introduced at the beginning of the paper. We study this random distribution
under the assumption  that $X_1,\ldots,X_n$ are an i.i.d.\ sample from the mixture density
$p_{G_0}=f\ast G_0$, for a given probability distribution $G_0$. 
We assume that $G_0$ is supported in a compact interval $[-a,a]$, and that
the base measure $\alpha$ of the Dirichlet prior in (i) is concentrated on this interval
with a Lebesgue density bounded away from 0 and $\infty$.  

\begin{theorem}
\label{thm-rate-wasserstein}
If $G_0$ is supported on $[-a, a]$ with $f$ being Laplace kernel and $\alpha$ has support $[-a,a]$ with 
Lebesgue density bounded away from 0 and $\infty$, then for every 
$k\ge 1$, there exists a constant $M$ such that
    \begin{equation}
        \label{eqn-main-results-rate-wasserstein}
        \Pi\bigl( G: W_k(G, G_0) \ge M n^{-3/(8k+16)}(\log n)^{(k+7/8)/(k+2)}| X_1,\ldots, X_n) \rightarrow 0,
    \end{equation}
 in $P_{G_0}$-probability. 
\end{theorem}

The rate for the Wasserstein metric $W_k$ given in the theorem deteriorates
with increasing $k$, which is perhaps not unreasonable as the Wasserstein
metrics increase with $k$. The fastest rate is $n^{-1/8}(\log n)^{5/8}$, and is obtained
for $W_1$.

\begin{theorem}
\label{thm-rate-mixture}
If $G_0$ is supported on $[-a, a]$ with $f$ being Laplace kernel and $\alpha$ has support $[-a,a]$ with 
Lebesgue density bounded away from 0 and $\infty$, then there exists 
a constant $M$ such that 
\begin{equation}
\Pi_n\bigl( G: h(p_G , p_{G_0}) \ge M (\log n/n)^{3/8}| X_1,\ldots, X_n\bigr) \rightarrow 0,
        \label{eqn-main-results-rate-mixture-h}
\end{equation}
in $P_{G_0}$-probability. Furthermore, for every $q\in[ 2,\infty)$ there exists $M_q$ such that 
\begin{equation}
\Pi_n\bigl( G: \|p_G -p_{G_0}\|_q \ge M_q (\log n/n)^{(q+1)/(q(q+2))}| X_1,\ldots, X_n\bigr) \rightarrow 0,
 \label{eqn-main-results-rate-mixture}
\end{equation}
in $P_{G_0}$-probability. 
\end{theorem}

The rate for the $L_q$-distance given in \eqref{eqn-main-results-rate-mixture}
deteriorates with increasing $q$. For $q=2$
it is the same as the rate $(\log n/n)^{3/8}$ for the Hellinger distance.

In both theorems the mixing distributions are assumed to be supported on a fixed compact set.
Without a restriction on the tails of the mixing distributions, no rate is possible. The assumption of a compact support
ensures that the rate is fully determined by the complexity of the mixtures, and not their tail behaviour.

\section{Finite Approximation}
\label{sec-approximation} 

In this section we show that a general mixture $p_G$ can be approximated
by a mixture with finitely many components, where the number of components
depends on the accuracy of the approximation, the distance used, and the
kernel $f$. We first consider approximation with respect to the $L_q$-norm,
which applies to mixtures $p_G=f\ast G$, for a general kernel $f$, and next 
approximation with respect to the Hellinger distance for the case that $f$
is the Laplace kernel. 
The first result generalizes 
the result of  \cite{ghosal_entropies_2001} for normal mixtures.  Also see
\cite{scricciolo}.

The result splits in two cases, depending on the tail behaviour of the
Fourier transform $\tilde f$ of $f$:
\begin{enumerate}
\item[-\emph{ordinary smooth $f$}\ :] $\limsup_{|\lambda|\rightarrow\infty} 
\bigl|\tilde{f}(\lambda)\bigr| |\lambda|^\beta<\infty$, for some $\beta>1/2$.
\item[-\emph{supersmooth $f$}\qquad:] $\limsup_{|\lambda|\rightarrow\infty}  \bigl|\tilde{f} (\lambda) \bigr|
  e^{|\lambda|^\beta}  <\infty$, for some $\beta>0$.
\end{enumerate}

\begin{lemma}
\label{lemma-find-discrete-measure}
Let $\varepsilon<1$ be sufficiently small and fixed.  For a probability measure $G$ on an interval $[-a, a]$ and $2 \le q \le \infty$, there exists a discrete measure $G'$ on $[-a, a]$ with at most $N$ support points in $[-a, a]$ such that
\begin{equation*}
        \norm{p_G - p_{G'}}_q \lesssim \varepsilon, 
 \end{equation*}
where 
\begin{enumerate}
\item[(i)] $N \lesssim \varepsilon^{- (\beta - p^{-1})^{-1}}$ if 
$f$ is ordinary smooth of order $\beta$, for $p$ and $q$ being conjugate ($p^{-1} + q^{-1} =1$).
\item[(ii)] $N \lesssim (\log \varepsilon^{-1} )^{\max(1,\beta^{-1})}$ if
$f$ is supersmooth of order $\beta$.
\end{enumerate}
\end{lemma}

\begin{proof}
The Fourier transform of $p_G$ is given by $ \tilde{f}\tilde{G} $, for 
$\tilde{G}(\lambda) = \int e^{\imath\lambda z}\,dG(z)$.  
Determine $G'$ so that it possesses the same moments as $G$ up to order $k-1$, i.e.\ 
    \begin{equation*}
        \int z^j d(G - G')(z) = 0, \quad \forall\ 0\le j \le k-1. 
    \end{equation*}
By Lemma~A.1 in \cite{ghosal_entropies_2001} $G'$ can be chosen to have at most $k$ support points.

Then for $G$ and $G'$ supported on $[-a, a]$, we have
\begin{align*}
        |\tilde{G}(\lambda) - \tilde{G'} (\lambda) | 
 & = \left|\int\Big( e^{\imath \lambda z} - \sum_{j=0}^{k-1}\frac{(\imath \lambda z)^j}{ j!}\Big)\, d(G-G')(z)\right|\\
   & \le \int \frac{| \imath \lambda z|^k}{ k!} d(G + G')(z) \le \Big( \frac{a e |\lambda|}{ k} \Big)^k. 
    \end{align*}
The inequality comes from 
$| e^{iy} - \sum_{j=0}^{k-1} {(iy)^j}/{j!}  | \le {|y|^k}/{ k!} \le{ (e|y|)^k}/{ k^k}$, for every $y\in\RR$.

Therefore, by Hausdorff-Young's inequality,
 \begin{align*}
\norm{p_G - p_{G'}}_q^p 
 & \le \frac{1}{ 2 \pi} \int |\tilde{f}(\lambda) |^p |\tilde{G} (\lambda) - \tilde{G'} (\lambda) |^p d\lambda \\
& \lesssim \int_{|\lambda| >M} |\tilde{f}(\lambda)|^p d\lambda + \int_{|\lambda| \le M} \Big( \frac{ea|\lambda|}{k} \Big)^{pk}\ d\lambda. 
 \end{align*}
We denote the first term in the preceding display by $I_1$ and the second term by $I_2$.  
It is easy to bound $I_2$ as:
    \[
        I_2 \asymp \left( \frac{ea}{k} \right)^{kp} \frac{M^{kp+1}}{kp +1}
        \lesssim \Bigl(\frac{eaM}{k}\Bigr)^{kp+1}\frac 1p. 
    \]
For $I_1$ we separately consider the supersmooth and ordinary smooth cases.    

In the supersmooth case with parameter $\beta$, we note that 
the function $(t^{\beta^{-1} -1})/e^{\delta t}$ is monotonely decreasing for $t\ge pM^{\beta}$,
when $\delta \ge (\beta^{-1} -1 ) / (pM^{\beta})$.  Thus, for large $M$,
    \begin{align*}
I_1 & \lesssim \int_{|\lambda|>M } e^{-p|\lambda|^{\beta}} d\lambda 
 = \frac{2}{\beta p^{\beta^{-1}}} \int_{t>pM^{\beta}} e^{-t}  t^{\beta^{-1} -1}  dt \\
& \le\frac{2}{\beta p^{\beta^{-1}}} \int_{t>pM^{\beta}}e^{-(1-\delta)t}\,dt \frac{(pM^\beta)^{\beta^{-1}-1}}{ e^{\delta p M^{\beta}}} 
 = \frac{2}{ 1- \delta} \frac{1}{ \beta p} e^{-p M^{\beta}} M^{1- \beta}, 
    \end{align*}
where the bound is sharper if $\delta$ is smaller. Choosing the minimal value of $\delta$,
we obtain 
 \begin{align*}
        I_1 & \lesssim \frac{1}{ 1 - (\beta^{-1} -1)/(p M^{\beta})} \frac{1}{\beta p} e^{-p M^{\beta}} M^{1- \beta}
%%        & = \frac{M}{ M^{\beta} + (1- \beta^{-1})/p} e^{-p M^{\beta}} \\
        \lesssim M^{1-\beta} e^{-pM^{\beta}}, 
    \end{align*}
for $M$ sufficiently large. 
We next choose $ M = 2 \left(\log ({1}/{\varepsilon})  \right)^{\frac{1}{\beta}}$ in order to
ensure that $I_1 \le \varepsilon^p$. Then  $I_2\lesssim\varepsilon^p$ if $k\ge 2ea M$ and 
$2^{-kp}\le \varepsilon^p$. This is satisfied if $k = 2( \log \varepsilon^{-1})^{\max(\beta^{-1},1)} $.

In the ordinary smooth case with smoothness parameter $\beta$, we have the  bound
  \begin{equation*}
        I_1 \lesssim \int_{\lambda > M} |\lambda|^{-\beta p} d\lambda  \lesssim \left(\frac{1}{M}\right)^{\beta p -1}. 
    \end{equation*}
We choose  $M = ( {1}/{\varepsilon})^{-(\beta  -1/p)^{-1}}$ to render the right side 
equal to $\varepsilon^p$. Then $I_2 \lesssim \varepsilon^p$ if 
$k = 2 \varepsilon^{- (\beta-1/p)^{-1}}$.
\end{proof}

The number of support points in the preceding lemma is increasing in $q$ and decreasing
in $\beta$. For approximation in the $L_2$-norm ($q=2$), the number of support points is of order
$\varepsilon^{-1/(\beta-1/2)}$, and this reduces to 
$\varepsilon^{-2/3}$ for the Laplace kernel (ordinary smooth with $\beta=2$). 
The exponent $\beta - 1/2$ can be interpreted as (almost) the Sobolev smoothness of $p_G$, since,
for  $\alpha < \beta - 1/2$,
\[
    \int (1 + |\lambda|^2)^\alpha |\tilde{p}_G(\lambda)|^2 d \lambda \lesssim \int (1 + |\lambda|^2)^\alpha |\tilde{f}(\lambda)|^2 d \lambda <\infty.  
\]
We do not have a compelling intuition for this correspondence.

The Hellinger distance is more sensitive to areas where the densities are close to
zero. This causes that the approach in the preceding lemma does not give sharp results. The
following lemma does, but is restricted to the Laplace kernel.

\begin{lemma}
\label{lemma-bound-hellinger-discrete}
For a probability measure $G$ supported on $[-a, a]$ there exists a discrete measure $G'$ with 
at most $N \asymp \varepsilon^{- 2/3}$ support points such that for $p_{G} = f \ast G$ and 
$f$ the Laplace density 
    \begin{equation*}
     h(p_{G}, p_{G'})  \le \varepsilon. 
    \end{equation*}
\end{lemma}

\begin{proof}
Since $p_{G}(x) \ge f(|x|+a)=e^{-a}e^{-|x|}/2$,
 for every $x$ and probability measure $G$ supported on $[-a,a]$,
 the Hellinger distance between Laplace mixtures satisfies 
\begin{align*}
 h^2(p_{G}, p_{G'}) & \le  \int \frac{ (p_{G} - p_{G'})^2 }{ p_{G}+ p_{G'}}(x)\, dx 
% \le \int \frac{(p_{G'}(x) -p_{G}(x))^2 }{2\nu(|x| +a)}  dx 
 \le  e^a \int (p_{G'}(x) -p_{G}(x))^2 e^{|x|}\,  dx.
    \end{align*}
If we write $q_{G}(x) = p_{G}(x) e^{|x|/2}$, and $\tilde q_G$ for the corresponding Fourier
transform, then the integral in the right side is equal to 
$(1/ 2\pi) \int |\tilde{q}_{G'} - \tilde{q}_{G}|^2(\lambda)\, d\lambda$,
by Plancherel's theorem. By an explicit computation we obtain
\begin{align*}
\tilde{q}_{G} (\lambda) &  = \frac{1}{2} \int \int e^{\imath \lambda x} e^{-|x-z| + |x|/2}\, dx\, dG(z) 
%        & = \frac{1}{2} \int \Big( \int_{-\infty}^0 e^{\imath \lambda x} e^{-z + x - x/2} dx + \int_{0}^z e^{\imath \lambda x} e^{-z + x + x/2} dx \\
 %       & \qquad \qquad + \int_z^{\infty} e^{\imath \lambda x} e^{z-x + x/2}dx \Big) d{G}(z) \\
         = \frac{1}{2} \int r(\lambda,z)\, d{G}(z),
    \end{align*}
where $r(\lambda,z)$ is given by 
\begin{align}
r(\lambda, z) &  = \frac{e^{-z}}{\imath \lambda + 1/2} + e^{-z} \frac{e^{(\imath \lambda + 3/2)z} -1 }{\imath \lambda + 3/2} - \frac{e^{(\imath \lambda +1/2)z}}{\imath \lambda - 1/2} \label{eqn-approx-kl-explicit-r} \nonumber \\
        & = \frac{e^{-z}}{(\imath\lambda+1/2)(\imath\lambda+3/2)}
 - \frac{2 e^{\imath \lambda z} e^{z/2}}{(\imath\lambda + 3/2) (\imath \lambda - 1/2)}. 
    \end{align}
Now let  $G'$  be a discrete measure on $[-a,a]$ such that
\begin{align*}
        & \int e^{-z}\, d({G'}-G)(z) = 0,\qquad 
         \int e^{z/2} z^j\,d({G'}- G)(z) = 0, \quad \forall\  0\le j \le k-1 .
    \end{align*}
By Lemma~A.1 in \cite{ghosal_entropies_2001} 
$G'$ can be chosen to have at most $k+1$ support points.

By the choice of $G'$ 
the first term of $r(\lambda,z)$
gives no contribution to the difference $\int r(\lambda, z)\, d ({G'} - G)(z)$. 
As the second term of $r(\lambda,z)$ is for large $|\lambda|$ bounded
in absolute value by a multiple of $|\lambda|^{-2}$, it follows that
\begin{align*}
I_2:=\int_{|\lambda|>M}  \left|  \int r(\lambda, z)\, d ({G'} - G)(z) \right|^2\, d \lambda
 \lesssim \int_{\lambda >M} \lambda^{-4} d\lambda \asymp M^{-3}.
\end{align*}
By the choice of $G'$ in the second term of $r(\lambda,z)$ we can replace $e^{i\lambda z}$  
by $e^{\imath\lambda z}-\sum_{j=0}^k (\imath\lambda z)^j/j!$ again without
changing the integral $\int r(\lambda, z)\, d ({G'} - G)(z)$.
It follows that 
\begin{align*}
I_1&:= \int_{|\lambda|\le M}  \left|  \int r(\lambda, z)\, d ({G'} - G)(z) \right|^2\, d \lambda\\
& \le \int_{|\lambda| \le M}  \left|\frac{2}{(\imath\lambda+1/2)(\imath\lambda+3/2)}\right|^2 \left| \int e^{z/2}\Big[e^{\imath \lambda z} - \sum_{j=0}^k (\imath \lambda z)^j\Bigr]\,  d({G'}-G)(z)\right|^2 d \lambda \\
& \qquad \qquad \qquad \lesssim \int_{0}^{M} \frac{ (z\lambda)^{2k}}{(k!)^2}\, d\lambda 
\lesssim \frac{(aeM)^{2k+1}}{k^{2k+1}}.
\end{align*}
It follows, by a similar argument as in the proof of Lemma~\ref{lemma-find-discrete-measure},
that we can reduce both  $I_1$ and $I_2$ to $\varepsilon^2$
by choosing and $M\asymp\varepsilon^{-2/3}$ and $k =2aeM$.
\end{proof}

\section{Entropy}
\label{sec-entropy-mixture}

We study the covering numbers of the class of mixtures $p_G=f\ast G$,
where $G$ ranges over the collection $\mathcal{M}[-a,a]$ of all probability measures on $[-a, a]$.  
We present a bound for any $L_q$-norm and general kernels $f$, and
a bound for the Hellinger distance that is specific to
the Laplace kernel. 

\begin{proposition}
\label{lemma-entropy-estimate-lq}
If both $\norm{f}_q$ and $\norm{f'}_q$ are finite and $\tilde{f}$ 
has ordinary smoothness $\beta$, then, for $p_G=f\ast G$, and any $q\ge 2$,
\begin{equation}
\log N\bigl(\varepsilon, \{p_G:  G\in\mathcal{M}[-a,a]\}, \norm{\cdot}_q\bigr)
 \lesssim \left(  \frac{1}{\varepsilon}\right)^{\frac{1}{ \beta - 1 + 1/q}} \log \Big( \frac{1}{\varepsilon} \Big). 
        \label{eqn-entropy-estimate-lq}
\end{equation}
\end{proposition}

\begin{proof}
Consider an $\varepsilon$-net of $\calG_a=\{p_G: G\in\mathcal{M}[-a,a]\}$ 
by constructing $\mathcal{I}$ the collection of
  all $p_G$'s such that the mixing measure $G \in \mathcal{M}[-a, a]$ is discrete and has at most $N
  \le D\varepsilon^{-(\beta -1 + q^{-1})^{-1}}$ support points for some proper constant $D$.  

In light of the approximation Lemma~\ref{lemma-find-discrete-measure}, 
the set of all mixtures $p_G$ with $G$ a discrete probability measure with
$N\lesssim \varepsilon^{-(\beta -1 + q^{-1})^{-1}}$ support points forms an $\varepsilon$-net
over the set of all mixtures $p_G$ as in the lemma. It suffices to construct an
$\varepsilon$-net of the given cardinality over this set of discrete mixtures.

By Jensen's inequality and Fubini's theorem,
\begin{align*}
\norm{f(\cdot- \theta ) - f }_q 
& = \left( \int \left| \theta \int_0^{1} f'(x- \theta s)\, ds \right|^q\, dx \right)^{1/q} 
%   & \le \theta \left( \int \int_0^{1} | f'(x- \theta  s )|^q ds \right)^{1/q} \\
 %   & \le \theta \left( \int_0^1 \int |f' (x - \theta s) |^q dx ds\right)^{1/q} \\
 \le \norm{f'}_q \theta.
\end{align*}
Furthermore, for any probability vectors $p$ and $p'$ and locations $\theta_i$,
\begin{align*}
   \left\|\sum_{i=1}^N p_i f(\cdot-\theta_i) - \sum_{i=1}^N p'_i f(\cdot-\theta_i)\right\|_q 
%& = \left\|\sum_{i=1}^N (p_i - p'_i) f(x- \theta_i)\right\|_q \\
    & \le \sum_{i=1}^N |p_i - p'_i| \norm{f(\cdot-\theta_i)}_q 
 = \norm{f}_q \|p-p'\|_1. 
\end{align*}
Combining these inequalities, we see that for two discrete probability measures
$G = \sum_{i=1}^N p_i \delta_{\theta_i}$ and $G' =\sum_{i=1}^N p_i' \delta_{\theta'_i}$,
\begin{align}
 \norm{p_G - p_{G'}}_q 
& \le  \norm{f'}_q \max_i|\theta_i-\theta_i'|+ \norm{f}_q\|p-p'\|_1.
 \label{eqn-bound-mixtures-by-relocating}
\end{align}
Thus we can construct an $\varepsilon$-net over the discrete mixtures by
relocating the support points $(\theta_i)_{i=1}^N$ 
to the nearest points $(\theta'_i)_{i=1}^N$ in a $\varepsilon$-net on $[-a,a] $, and relocating
the weights $p$ to the nearest point $p'$ in an $\varepsilon$-net for the $l_1$-norm
over the $N$-dimensional $l_1$-unit simplex. This gives a set of at most 
\begin{equation*}
\left( \frac{2a}{ \varepsilon} \right)^N  \left( \frac{5}{\varepsilon}  \right)^N 
\sim  \left( \frac{10a}{ \varepsilon^2 } \right)^N
\end{equation*}
measures $p_G$ (cf. Lemma~A.4 of \cite{ghosal_posterior_2007} for the
entropy of the  $l_1$-unit simplex). This gives the bound of the lemma.
\end{proof}

\begin{proposition}
\label{LemmaEntropyHellinger}
For $f$ the Laplace kernel and $p_G=f\ast G$, 
\begin{equation}
    \log N\bigl(\varepsilon, \{p_G: G\in \mathcal{M}[-a,a]\}, h\bigr) 
\lesssim \varepsilon^{-3/8}\log (1/ \varepsilon). 
    \label{eqn-entropy-estimate-h}
\end{equation}
\end{proposition}

\begin{proof}
Because the function $\sqrt f$ is absolutely continuous with derivative 
$x\mapsto -2^{-3/2}e^{-|x|/2}\sgn(x)$,
we have by Jensen's inequality and Fubini's theorem that 
\begin{align*}
h^2\bigl(f, f(\cdot-\theta)\bigr) 
& =\int \Bigl(\theta \int_0^1 -2^{-3/2}e^{-|x-\theta s|/2}\sgn(x-\theta s)\, ds\Bigr)^2 \, dx \\
& \le \theta^2 \int_0^1 \int e^{-|x-\theta s|}\, dx\, ds=2\theta^2. 
 \end{align*}
 It follows that $h\bigl(f, f(\cdot-\theta) \bigr) \lesssim \theta$.

By convexity of the map
 $(u,v)\mapsto (\sqrt u-\sqrt v)^2$, we have
$$\Bigl|\sqrt{\sum_{i}p_if(\cdot-\theta_i)}-\sqrt{\sum_{i}p_if(\cdot-\theta_i')}\Bigr|^2
\le \sum_{i}p_i\bigl[\sqrt{f(\cdot-\theta_i)}-\sqrt{f(\cdot-\theta_i')}\bigr]^2.$$
By integrating this inequality we see that the densities $p_G$ and $p_{G'}$ with mixing
 distributions $G = \sum_{i=1}^N p_i \delta_{\theta_i}$ and $G' = \sum_{i=1}^N p_i
 \delta_{\theta'_i}$ satisfy $h^2(p_G, p_{G'})\lesssim \sum p_i|\theta_i-\theta_i'|^2\le
 \|\theta-\theta'\|_\infty^2$. 

Furthermore, for distributions $G = \sum_{i=1}^N p_i \delta_{\theta_i}$ and $G' =\sum_{i=1}^N p'_i
\delta_{\theta_i}$ with the same support points, but different weights, we have 
\begin{align*}
 h^2(p_G, p_{G'}) & 
\le \int \frac{\big( \sum_{i=1}^N (p_i - p'_i) f(x- \theta_i) \big)^2}{ \sum_{i=1}^N (p_i + p'_i) f(x- \theta_i)} dx\\
& \le \int \big( \sum_{i=1}^N |p_i - p'_i| \big)^2\frac{ f^2(|x| -a) }{2 f(|x| +a) } dx 
\lesssim \norm{p- p'}_{1}^2. 
    \end{align*}
Therefore the bound follows by arguments similar as in the proof of 
Proposition~\ref{lemma-entropy-estimate-lq}, where presently we use
Lemma~\ref{lemma-bound-hellinger-discrete} to determine suitable
finite approximations.
\end{proof}

The map $G\mapsto p_G=f\ast G$ is one-to-one as soon as the characteristic function
of $f$ is never zero. Under this condition we can also
view the Wasserstein distance on the mixing distribution as a distance on the
mixtures. Obviously the covering numbers are then free of the kernel.

\begin{proposition}
\label{PropositionEntropyEstimateW}
For any $k\ge 1$, and any sufficiently small $\varepsilon>0$,
\begin{equation}
\log N\bigl(\varepsilon, \mathcal{M}[-a,a], W_k\bigr) 
\lesssim \Bigl(\frac1\varepsilon\Bigr)\log (1/ \varepsilon). 
    \label{eqn-entropy-estimate-W}
\end{equation}
\end{proposition}

The proposition is a consequence Lemma~\ref{LemmaEntropyWGeneral}, below,
which applies to the set of all Borel probability measures on a general metric space $(\Theta,\rho)$
(cf. \cite{nguyen_convergence_2013}).

\begin{lemma}
\label{lemma:countable-partitions}
For any probability measure $G$ concentrated on countably many disjoint sets $\Theta_1,\Theta_2, \ldots$ 
and probability measure $G'$ concentrated on disjoint sets $\Theta_1',\Theta'_2,\ldots$,
\begin{displaymath}
        W_k(G,G') \le \sup_i \sup_{\theta_i \in \Theta_i, \theta'_i \in \Theta'_i} \rho(\theta_i, \theta'_i) + \diam(\Theta) \Big( \sum_i |G(\Theta_i) - G'(\Theta'_i)| \Big)^{1/k}. 
    \end{displaymath}
In particular,
\begin{displaymath}
        W_k \Big(\sum_i p_i \delta_{\theta_i}, \sum_i p'_i \delta_{\theta'_i} \Big) \le \max_i \rho(\theta_i, \theta'_i) + \diam (\Theta) \norm{p-p'}_1^{1/k}. 
    \end{displaymath}
\end{lemma}

\begin{proof}
  For $p_i = G(\Theta_i)$ and $p'_i = G'(\Theta'_i)$ divide the interval $[0, \sum_i p_i \wedge
  p'_i]$ into disjoint intervals $I_i$ of lengths $p_i \wedge p'_i$.  We couple variables
  $\bar{\theta}$ and $\bar{\theta}'$ by an auxiliary uniform variable $U$.  If $U \in I_i$, then
  generate $\bar{\theta} \sim G(\cdot|\Theta_i)$ and $\bar{\theta}' \sim G'(\cdot|\Theta'_i)$.
  Divide the remaining interval $[\sum_i p_i \wedge p_i', 1] $ into intervals $J_i$ of lengths $p_i -
  p_i \wedge p'_i$ and, separately, intervals $J'_i$ of length $p'_i - p_i\wedge p'_i$.  If $U\in
  J_i$, then generate $\bar{\theta}\sim G(\cdot| \Theta_i) $ and if $U\in J'_i$, then generate
  $\bar{\theta}' \sim G'(\cdot| \Theta'_i)$. Then $\bar\theta$ and $\bar\theta'$ have marginal
distributions $G$ and $G'$, and 
        \begin{displaymath}
            \bbE \rho^k ( \bar{\theta}, \bar{\theta}') \le \bbE \big[ \rho^k(
\bar{\theta}, \bar{\theta}') \text{1}_{U\le \sum_i p_i \wedge p'_i} \big] + 
\diam (\Theta)^k \bbP \big( U > \sum_i p_i \wedge p'_i \big).
        \end{displaymath}
        The first term is bounded by the $k$-th power of the first term of the lemma, while the probability in the second term is equal to $1- \sum_i p_i \wedge p'_i = \sum_i |p_i - p'_i|/2$. 
\end{proof}

\begin{lemma}
\label{LemmaEntropyWGeneral}
For the set $\mathcal{M}(\Theta)$ of all Borel probability measures on 
a metric space $(\Theta,\rho)$, any $k\ge 1$, and $0< \varepsilon< \min\{ 2/3, \diam(\Theta)\}$, 
    \begin{displaymath}
        N\bigl(\varepsilon, \mathcal{M}(\Theta), W_k\bigr) 
\le \Big( \frac{ 4 \diam(\Theta)}{\varepsilon}\Big)^{kN(\varepsilon, \Theta, \rho)}. 
    \end{displaymath}
    \label{lemma:bound-packing-numbers}
\end{lemma}

\begin{proof}
For a minimal $\varepsilon$-net over $\Theta$ of $N=N(\varepsilon,\Theta,\rho)$ points, let
  $\Theta = \cup_i \Theta_i$ be the partition obtained by assigning each $\theta$ to a closest
  point.  For any $G$ let $G_\varepsilon = \sum_i G(\Theta_i)\delta_{\theta_i}$, for arbitrary but
  fixed $\theta_i\in \Theta_i$.  Since $W_k(G,G_\varepsilon) \le \varepsilon$ by
  Lemma~\ref{lemma:countable-partitions}, we have 
$N(2\varepsilon, \mathcal{M}(\Theta), W_k) \le N(\varepsilon,  \mathcal{M}_\varepsilon, W_k) $, 
for $\mathcal{M}_\varepsilon$ the set of all $G_\varepsilon$.  We next form the measures
$G_{\varepsilon,p} = \sum_i p_i \delta_{\theta_i}$ for $(p_1,\ldots, p_N)$ ranging over
an  $(\varepsilon/\diam(\Theta))^k$-net for the $l_1$-distance over the $N$-dimensional unit 
simplex. By   Lemma~\ref{lemma:countable-partitions} every $G_\varepsilon$ is
within $W_k$-distance of some $G_{\varepsilon,p}$. Thus 
$N(\varepsilon, \mathcal{M}_\varepsilon, W_k)$ is bounded from above by the number of
  points $p$, which is bounded by $(4\diam(\Theta)/\varepsilon)^{k N}$
(cf.\ Lemma~A.4 in \cite{ghosal_convergence_2000}).
\end{proof}

\section{Prior mass}
\label{sec-prior-mass}
This main result of this section is the following proposition, which gives
a lower bound on the prior mass of the prior (i)-(iv) in a neighbourhood of 
a mixture $p_{G_0}$. 

\begin{proposition}
\label{PropositionPriorMass}
If $\Pi$ is the Dirichlet process $\DP(\alpha)$ with base measure $\alpha$ that has a Lebesgue
density bounded away from 0 and $\infty$ on its support $[-a,a]$, and $f$ is the Laplace kernel,
then for every sufficiently small
$\varepsilon>0$ and every probability measure $G_0$ on $[-a,a]$,
\begin{equation*}
\log \Pi \Big(G: K(p_G, p_{G_0}) \le \varepsilon^2, K_2(p_G, p_{G_0}) \le \varepsilon^2 \Big) 
\gtrsim  \Big( \frac{1}{\varepsilon}  \Big)^{2/3} \log\Big( \frac{1}{\varepsilon}\Big).
\end{equation*}
\end{proposition}

\begin{proof}
By Lemma~\ref{lemma-bound-hellinger-discrete} there
exists a discrete measure $G_1$ with $N\lesssim \varepsilon^{-2/3}$ support points
such that $h (p_{G_0} , p_{G_1}) \le \varepsilon$.  We
may assume that the support points of $G_1$ are at least $2\varepsilon^2$-separated.  
If not, we take a maximal $2\varepsilon^2$-separated set in the support 
points of $G_1$, and replace $G_1$ by the discrete measure obtained
by relocating the masses of $G_1$ to the nearest points in the
  $2\varepsilon^2$-net.  Then  $h(p_{G_1}, p_{G'_1}) \lesssim \varepsilon^2$,
as seen in the proof of Proposition~\ref{LemmaEntropyHellinger}.
  
Now by Lemmas~\ref{lemma-bound-h-with-L2} and~\ref{lemma-L1-discrete}, if
$G_1 = \sum_{i=1}^N p_j \delta_{z_j}$, with the support points $z_j$ 
at least $2 \varepsilon^2$-separated,
\begin{align*}
\bigl\{G:  \max( K, K_2) (p_{G_0},p_G)< d_1\varepsilon^2 \bigr\} 
& \supset \bigl\{G: h(p_{G_0}, p_G) \le 2 \varepsilon \bigr\} \\
& \supset \bigl\{G: h(p_{G_1}, p_G) \le  \varepsilon \bigr\} \\
& \supset \bigl\{ G: \norm{p_G - p_{G_1}}_{1} \le d_2 \varepsilon^2 \bigr\} \\
& \supset \bigl\{G: \sum_{j=1}^N \big|G[z_j - \varepsilon^2, z_j + \varepsilon^2] - p_j\big| \le \varepsilon^2 \bigr\}.
\end{align*}
By Lemma~A.$2$ of \cite{ghosal_entropies_2001}, since
the base measure $\alpha$ has density bounded away from zero and infinity
on $[-a, a]$ by assumption, we have 
\begin{equation*}
 \log \Pi \left(G: \sum_{j=1}^N \big| G[z_j - \varepsilon^2, z_j + \varepsilon^2] - p_j\big| \le
 \varepsilon^2   \right) 
\gtrsim -N\log\Big( \frac{1}{\varepsilon} \Big) 
\end{equation*} 
The lemma follows upon combining the preceding.
\end{proof}

\begin{lemma}
\label{lemma-L1-discrete}
If $G'= \sum_{j=1}^{N} p_i \delta_{z_j}$ is a probability measure  supported on
points $z_1,\ldots, z_N$ in $\RR$ with $|z_j-z_k| > 2\varepsilon$ for $j \neq k$,
then for any probability measure $G$ on $\mathbb{R}$ and kernel $f$ ,
\begin{equation*}
\norm{p_G-p_{G'}}_1 \le 2\|f'\|_1 \varepsilon 
+2 \sum_{j=1}^N \big| G[z_j-\varepsilon, z_j+\varepsilon] - p_j\big|.
\end{equation*}
\end{lemma}

\beginskip
\begin{proof}
   First we claim, for $\delta$ sufficiently small, 
    \begin{equation*}
        \int |f(x) - f(x+\delta) | dx \lesssim \delta.
    \end{equation*}
    By Fubini's theorem
    \begin{equation}
        \begin{split}
            \int |f(x) - f(x+\delta) | dx & \le \int \int_{x}^{x+\delta} |f'(s)|\, ds\, dx \\
            & = \delta \int |f'(s)| ds \lesssim \delta.  
        \end{split}
        \label{eqn-bound-L1-f-x+delta-f}
    \end{equation}
   
    We bound $|p_G - p_{G'}|$ as follows
    \begin{align*}
        |(p_G- p_{G'})(x)| & = \left|  \int f(x-z) dG(z) - \sum_{j=1}^N f(x-z_j) p_j \right| \\
        & \le \left| \sum_{j=1}^N \int_{|z-z_j|\le \varepsilon} f(x-z) d G(z) - \sum_{j=1}^N f(x-z_j)p_j \right| \\
        & \quad + \left| \int_{z: |z-z_j|>\varepsilon\forall j} f(x-z) dG(z)  \right|\\
        & \le \left| \sum_{j=1}^N \int_{|z-z_j|\le \varepsilon} \big(f(x-z) - f(x-z_j) \big) dG(z) \right| \\
        & \quad + \left|  \sum_{j=1}^N f(x-z_j) G[z_j-\varepsilon, z_j+\varepsilon]  - \sum_{j=1}^N f(x-z_j) p_j \right| \\
        & \quad + \int_{z: |z-z_j|>\varepsilon\forall j} f(x-z) dG(z)  \\
        & \le \sum_{j=1}^N \int_{|z-z_j|\le \varepsilon} |f(x-z) - f(x-z_j)| dG(z) \\ 
        & \quad + \sum_{j=1}^N f(x-z_j) \big| p_j - G[z_j-\varepsilon, z_j+\varepsilon] \big| \\ 
        & \quad + \int_{z: |z-z_j|>\varepsilon\forall j} f(x-z) d G(z).
    \end{align*}
    Denote the three terms in the preceding display as $I_1(x)$, $I_2(x)$ and $I_3(x)$.  

    Fix $z$ and $z_j$, by (\ref{eqn-bound-L1-f-x+delta-f}) we have 
    \begin{align*}
        \int | f(x-z) - f(x- z_j) | dx \lesssim |z-z_j|. 
    \end{align*}
    By Fubini's theorem, we have 
    \begin{align*}
        \int I_1(x) dx &  = \sum_{j=1}^N \int_{|z-z_j| \le \varepsilon} \int \big| f(x-z) - f(x-z_j) \big| dx dG(z) \\
        & \lesssim \sum_{j=1}^N \int_{|z-z_j| \le \varepsilon} |z-z_j| d G(z) \le \varepsilon. 
    \end{align*}
    Note that $f(x)$ integrates to $1$, 
    \begin{equation*}
        \int I_2(x) dx \le \sum_{j=1}^N \big|p_j - G[z_j-\varepsilon, z_j+\varepsilon]\big|. 
    \end{equation*}
    noticing $(p_j)_j$ sums up to $1$, 
    \begin{align*}
        \int I_3(x)dx & \le G(z: |z-z_j|>\varepsilon \forall j) \\
        & = 1-\sum_{j=1}^N G[z_j-\varepsilon, z_j+\varepsilon] \\
        & \le \sum_{j=1}^N \big|G[z_j-\varepsilon, z_j+\varepsilon] - p_j\big|.
    \end{align*}
    Combing the above displays, we have the desired result. 
\end{proof}
\endskip

\begin{lemma}
\label{lemma-bound-h-with-L2}
If $G$ and $G'$ are probability measures on $[-a,a]$, and $f$ is the Laplace kernel, then 
\begin{align}
\label{eqn-bound-hellinger-l2}
h^2(p_G, p_{G'})& \lesssim \norm{p_G - p_{G'}}_{2},\\
\max \big(K(p_{G}, p_{G'}), K_2(p_{G}, p_{G'})\big) & \lesssim h^2(p_{G}, p_{G'}).
\label{lemma-bound-kl-hellinger}
\end{align}
\end{lemma}

\begin{proof}[Proofs]
The first lemma is a generalization of Lemma~4 in \cite{ghosal_posterior_2007}
from normal to general kernels, and is proved in the same manner.

In view of the shape of the Laplace kernel, it is easy to see that for 
$G$ compactly supported on $[-a,a]$, 
  \begin{equation*}
        f(|x| +a) \le p_G(x) \le f(|x| - a), 
    \end{equation*}
We bound the squared Hellinger distance as follows:
    \begin{align*}
        h^2(p_G, p_{G'}) & \le \int \frac{(p_G - p_{G'})^2}{ p_G + p_{G'}} dx\\
        & \le \int_{|x| \le A} e^{A+a} (p_G - p_{G'} )^2 dx + \int_{|x| > A} (p_G + p_{G'}) dx \\
        & \lesssim e^a  \norm{p_G - p_{G'}}_2^2 e^A + e^{-A}. 
    \end{align*}
By the elementary inequality $ t + \frac{u}{t } \ge 2 \sqrt{u}$, for $u, t>0$,
we obtain \eqref{eqn-bound-hellinger-l2} upon choosing 
$A = \min(a,  \log \norm{p_G - p_{G'}}_2^{-1}- a/2) $. 

For the proof of the second assertion we first note that, if both $G$ and $G'$ 
are compactly supported on $[-a, a]$, 
    \begin{equation*}
        \frac{p_{G}(x)}{p_{G'}(x)} \le \frac{ f(|x| -a )}{ f (|x| +a) } \le e^{2a}.
    \end{equation*}
Therefore $\norm{p_{G}/p_{G'}}_{\infty} \le e^{2a}$, and \eqref{lemma-bound-kl-hellinger} follows
by Lemma~8 in \cite{ghosal_posterior_2007}. 
\end{proof}

\section{Proof of Theorem~\ref{thm-rate-wasserstein}}
\label{SectionProofsMain}
The proof is based on the following comparison between the Wasserstein and Hellinger
metrics. The lemma improves and generalizes Theorem~$2$ in \cite{nguyen_convergence_2013}.
Let $C_k$ be a constant such that the map $\varepsilon\mapsto \varepsilon[\log (C_k/\varepsilon)]^{k+1/2}$
is monotone on $(0,2]$.

\begin{lemma}
\label{LemmaCompareWh}
For probability measures $G$ and $G'$ supported on $[-a,a]$, and $p_G=f\ast G$ for
a probability density $f$ with $\inf_{\lambda}(1+|\lambda|^\beta)|\tilde f(\lambda)|>0$, and any $k\ge 1$,
$$W_k(G,G')\lesssim h(p_G,p_{G'})^{1/(k+\beta)}\Bigl(\log \frac{C_k} {h(p_G,p_{G'})}\Bigr)^{(k+1/2)/(k+\beta)}.$$
\end{lemma}

\begin{proof}
By Theorem~6.15 in 
\cite{villani_optimal_2009} the Wasserstein distance $W_k(G,G')$ is bounded
above by a multiple of the $k$th root of $\int |x|^k\,d|G-G'|(x)$, where $|G-G'|$ is the
total variation measure of the difference $G-G'$. We apply this to the convolutions
of $G$ and $G'$ with the normal distribution $\Phi_\delta$ with mean 0 and variance $\delta^2$, to
find, for every $M>0$,
\begin{align*}
W_k(G\ast \Phi_\delta,&G'\ast\Phi_\delta)^k
\lesssim \int |x|^k\,\bigl|(G-G')\ast \phi_\delta(x)\bigr|\,dx\\
&\le\Bigl(\int_{-M}^M x^{2k}\,dx\,\int_{-M}^M \bigl|(G-G')\ast \phi_\delta(x)\bigr|^2\,dx\Bigr)^{1/2}\\
&\qquad\qquad+ e^{-M}\int_{|x|>M}|x|^ke^{|x|}\bigl|(G-G')\ast \phi_\delta(x)\bigr|\,dx\\
&\lesssim M^{k+1/2}\bigl\|(G-G')\ast \phi_\delta\bigr\|_2+e^{-M}e^{2|a|}\E e^{2|\delta Z|},
\end{align*}
where $Z$ is a standard normal variable. The number
$K_\delta:=e^{2|a|}\E e^{2|\delta Z|}$ is uniformly bounded by if $\delta\le \delta_k$, for some
fixed $\delta_k$. 

By Plancherel's theorem,
\begin{align*}
\bigl\|(G-G')\ast \phi_\delta\bigr\|_2^2
&=\int |\tilde G-\tilde G'|^2(\lambda)\tilde \phi_\delta^2(\lambda)\,d\lambda
=\int |\tilde f(\tilde G-\tilde G')|^2(\lambda) \frac{\tilde \phi_\delta^2}{|\tilde f|^2}(\lambda)\,d\lambda\\
&\lesssim \|p_G-p_{G'}\|_2^2\, \sup_\lambda \frac{\tilde \phi_\delta^2}{|\tilde f|^2}(\lambda)
\lesssim h^2(p_G,p_{G'}) \delta^{-2\beta},
\end{align*}
where we have again applied Plancherel's theorem, used that
the $L_2$-metric on uniformly bounded densities is bounded by
the Hellinger distance, and the assumption on the Fourier transform of $f$, which shows
that $(\tilde\phi_\delta/|\tilde f|)(\lambda)\lesssim
(1+|\lambda|^\beta)e^{-\delta^2\lambda^2/2}\lesssim \delta^{-\beta}$.

If $U\sim G$ is independent of $Z\sim N(0,1)$, then $(U,U+\delta Z)$ gives a coupling 
of $G$ and $G\ast \Phi_\delta$. Therefore the definition of the Wasserstein metric  gives that
$W_k(G,G\ast\Phi_\delta)^k\le \E |\delta Z|^k\lesssim \delta^k$. 

Combining the preceding inequalities with the triangle inequality we see that, for
$\delta\in (0,\delta_k]$ and any $M>0$,
$$W_k(G,G')^k
\lesssim M^{k+1/2}h(p_G,p_{G'})\delta^{-\beta}+e^{-M}+\delta^k.$$
The lemma follows by optimizing this over $M$ and $\delta$. 
Specifically, for $\varepsilon=h(p_G,p_{G'})$,
we choose $M=k/(k+\beta)\,\log(C_k/\varepsilon)$ and $\delta=(M^{k+1/2}\varepsilon)^{1/(k+\beta)}$.
These are eligible choices for 
$$\delta_k=\sup_{\varepsilon\in(0,2]}\Bigl[\frac k{k+\beta}\log\frac{C_k}\varepsilon\Bigr]^{(k+1/2)/(k+\beta)}
\varepsilon^{1/(k+\beta)},$$
which is indeed a finite number. In fact the supremum is taken at $\varepsilon=2$, by the assumption on $C_k$.
\end{proof}

For the Laplace kernel $f$ we choose $\beta=2$ in the preceding lemma, and then obtain
that $d(p_G, p_{G'}) \le h(p_G, p_{G'})$, for the ``discrepancy'' $d=\gamma^{-1}(W_k)$, and
$\gamma(\varepsilon)=D_k\varepsilon^{1/(k+\beta)}[\log (C_k/\varepsilon)]^{(k+1/2)/(k+\beta)}$
a multiple of the (monotone) transformation in the right side of the preceding lemma.
For small values of $W_k(G_1,G_2)$ we have
\begin{equation}
 \label{def-connect-G-PG}
 d(p_{G_1}, p_{G_2}) \asymp W_k^{k+2} ( G_1, G_2)\Bigl(\log \frac{1}{W_k(G_1,G_2)}\Bigr)^{-k-1/2}.
\end{equation}
As $k+2>1$ the discrepancy $d$ may not satisfy the triangle inequality,
but it does possess the properties (a)--(d) in the appendix, Section~\ref{Appendix}.  
The balls of the discrepancy $d$ are convex, as the Wasserstein metrics are convex  (see \cite{villani_optimal_2009}). 

It follows that Theorem~\ref{thm-rate-semi-metric} applies to obtain a
rate of posterior contraction relative to $d$ and hence relative to 
$W_k\sim d^{1/(k+2)}\bigl(\log (1/d)\bigr){}^{(k+1/2)/(k+2)}$.
We apply the theorem with $\mathcal{P}=\mathcal{P}_n$ equal to 
the set of mixtures $p_G=f\ast G$, as $G$ ranges over $\mathcal{M}[-a,a]$.
Thus \eqref{eqn-semi-metric-sieve} is trivially satisfied.
 
For the entropy condition \eqref{eqn-semi-metric-entropy} we have, 
by Proposition~\ref{PropositionEntropyEstimateW},
\begin{align*}
 \log N(\varepsilon, \mathcal{P}_n, d) & = \log N\Bigl(\varepsilon^{1/(k+2)}\Bigl(\log\frac1\varepsilon\Bigr)^{(k+1/2)/(k+2)},
 \mathcal{M}[-a, a], W_k) \\
&\lesssim \Bigl(\frac{1}{\varepsilon}\Big)^{1/(k+2)} \Bigl(\log\frac{1}{\varepsilon}\Bigr)^{1+(k+1/2)/(k+2)}.
\end{align*}
Thus  \eqref{eqn-semi-metric-entropy}
holds for the rate $ \varepsilon_n\gtrsim n^{-\gamma}$, for every $\gamma< (k+2)/(2k+5)$.

The prior mass condition \eqref{eqn-semi-metric-prior} is satisfied with the 
rate  $\varepsilon_n\asymp(\log n / n)^{3/8}$, in view of 
Proposition~\ref{PropositionPriorMass}.

Theorem~\ref{thm-rate-semi-metric}  yields a rate of contraction relative to $d$
equal to the slower of the two rates, which is $(\log n/n)^{3/8}$.
This translates into the rate for the Wasserstein distance as given
in Theorem~\ref{thm-rate-wasserstein}.

\section{Proof of Theorem~\ref{thm-rate-mixture}}
\label{SectionProofsMainTwo} 
We apply Theorem~\ref{thm-rate-semi-metric}, with $\mathcal{P} = \mathcal{P}_n $ the set
of all mixtures $p_G$ as $G$ ranges over $\mathcal{M}[-a,a]$.
For $d = h$ the rate follows immediately by combining
Propositions~\ref{lemma-entropy-estimate-lq} and~\ref{PropositionPriorMass}.

Since the densities $p_G$ are uniformly bounded by $1/2$, the $L_q$ distance
$\|p_G-p_{G'}\|_q$ is bounded above by a multiple of $h(p_G,p_{G'})^{2/q}$. 
We can therefore apply  Theorem~\ref{thm-rate-semi-metric} with the
discrepancy $d(p,p')=\|p-p'\|_q^{q/2}$. In view of Proposition~\ref{lemma-entropy-estimate-lq}
$$\log N\bigl(\varepsilon,\mathcal{P}_n,d\bigr)\lesssim \varepsilon ^{-2/(q+1)}\log
(1/\varepsilon).$$
Therefore the entropy condition  \eqref{eqn-semi-metric-entropy} is satisfied with 
$\varepsilon_n\asymp (\log n/n )^{(q+1)/(2q+4)}$. 
By Proposition~\ref{PropositionPriorMass} the prior mass condition is satisfied
for $\varepsilon_n\asymp(\log n/n)^{3/8}$. 
By Theorem~\ref{thm-rate-semi-metric} the rate of contraction relative to $d$ is
the slower of these two rates, which is the first. The rate relative to the $L_q$-norm
is the $(2/q)$th power of this rate.

\beginskip    
In the case of $L_2$, we apply Proposition~$2$ of \cite{golubev1992nonparametric} to show the rate is minimax.  If we take $\sigma(\lambda) = (i\lambda)^{3/2 - \delta}$ for any $\delta >0$, we have
    \begin{equation*}
        \int |\sigma(\lambda)|^{2+ \rho} |\tilde{p}_{G_0}(\lambda)|^2 d\lambda \le \int
        |\sigma(\lambda)|^{2+\rho} |\tilde{f}(\lambda)|^2 d\lambda < \infty
    \end{equation*}
for some $\rho>0$.  It is implicitly shown in \cite{golubev1992nonparametric} that the rate of square risk, i.e., \ the square of $L_2$-norm, is bounded below by $n^{-3/4}$.  This concludes our statement that the rate $n^{-3/8}$, up to a logarithm factor, is optimal in the minimax sense.    
\endskip

\section{Normal mixtures}
We reproduce the results on normal mixtures from \cite{ghosal_entropies_2001}, but in $L_2$-norm.  Note the normal kernel is supersmooth with $\beta=2$, by the approximation lemma, for any measure $G_1$ compactly supported on $[-a, a]$ we can always find a discrete measure $G_2$ with number of support points of order $N \asymp \log \varepsilon^{-1}$ such that $\norm{p_{G_1} - p_{G_2}}_2 \le \varepsilon$.  It is easy to establish 
\begin{equation*}
    h^2(p_{G_1}, p_{G_2}) \lesssim  \norm{p_{G_1} - p_{G_2}}_2.
\end{equation*}

Following the same procedure as before, assuming $G_0$ is the true measure, we obtain for prior mass condition
\begin{equation*}
    \log \Pi \left(G:  \max\Big( P_{G_0} \log \frac{p_{G_0}}{p_G}, P_{G_0} \Big( \log \frac{p_{G_0}}{p_G}\Big)^2 \Big) \le \varepsilon^2 \right) \gtrsim - \Big( \log \frac{1}{\varepsilon} \Big)^2,
\end{equation*}
Thus we obtain $\varepsilon_n = \log n / \sqrt{n}$. 

By Lemma~\ref{lemma-entropy-estimate-lq}, we have the following estimate for entropy condition
\begin{equation*}
    \log N( \varepsilon, \calG_a, \| \cdot \|_2 ) \lesssim \Big( \log \frac{1}{\varepsilon} \Big)^2,
\end{equation*}
this coincides with the estimate of prior mass condition, thus we obtain the rate of $\varepsilon_n = \log n / \sqrt{n}$ with respect to $L_2$-norm.  This is the same with what is obtained in \cite{ghosal_entropies_2001}, only in $L_2$-norm.  However we lose a $\sqrt{\log n}$-factor comparing to \cite{watson_estimation_1963}, which is $\sqrt{\log n / n}$.  

\section{Appendix: contraction rates relative to non-metrics} 
\label{Appendix}
The basic theorem  of  \cite{ghosal_convergence_2000} 
gives a posterior contraction rate in terms of a metric on densities
that is bounded above by the Hellinger distance. In the present situation we 
would like to apply this result to a power smaller than one of the Wasserstein metric,
which is not a metric. In this appendix we establish a rate of contraction which is 
valid for more general discrepancies.

We consider a general ``discrepancy measure'' $d$, which is a map 
$d: \mathcal{P}\times\mathcal{P}\to\RR$
on the product of the set of densities on a given measurable space and itself, which
has the properties, for some constant $C>0$:
\begin{itemize}
    \item[(a)] $d(x,y)\ge 0$;
    \item[(b)] $d(x,y) = 0$ if and only if $x=y$;
    \item[(c)] $d(x,y) = d(y,x)$;
    \item[(d)] $d(x,y) \le C \big(d(x,z) + d(y,z)\big)$.
\end{itemize}
Thus $d$ is a metric  except that the triangle inequality
is replaced with a weaker condition that incorporates a constant $C$, possibly
bigger than 1. Call a set of the form $\{x: d(x,y)<c\}$ a $d$-ball, and define
covering numbers $N(\varepsilon,\mathcal{P},d)$ relative to $d$ as usual.

Let $\Pi_n(\cdot|X_1,\ldots, X_n)$ be the posterior distribution of $p$ given 
an i.i.d.\ sample $X_1,\ldots, X_n$ from a density $p$ that is equipped with a
prior probability distribution $\Pi$.

\begin{theorem}
\label{thm-rate-semi-metric}
Suppose $d$ has the properties as given, the sets $\{p: d(p,p')<\delta\}$ are convex, and satisfies
$d(p_0,p)\le h(p_0,p)$, for every $p\in\mathcal{P}$.
Then $\Pi_n\bigl(d(p,p_0)>M \varepsilon_n | X_1,\ldots, X_n\bigr)\rightarrow 0 $ in $P_0^n$-probability
for any $\varepsilon_n$ such that $n \varepsilon_n^2 \rightarrow \infty$ and such that, 
for positive constants $c_1$, $c_2$ and sets $\mathcal{P}_n \subset \mathcal{P}$, 
 \begin{gather}
        \log N(\varepsilon_n, \mathcal{P}_n, d) \le c_1 n \varepsilon_n^2, \label{eqn-semi-metric-entropy} \\ 
\Pi_n(p: K(p_0, p) < \varepsilon_n^2, K_2(p_0, p) < \varepsilon_n^2 ) \ge e^{- c_2 n \varepsilon_n^2}, 
\label{eqn-semi-metric-prior} \\
        \Pi_n(\mathcal{P}- \mathcal{P}_n) \le e^{- (c_2+4)n \varepsilon_n^2}. \label{eqn-semi-metric-sieve}
\end{gather}
\end{theorem}

\begin{proof}
For every $\varepsilon > 4 C \varepsilon_n$, we have $\log N( C^{-1} \varepsilon /4 , \mathcal{P}_n, d) \le \log N(\varepsilon_n, \mathcal{P}_n , d) \le c_1 n \varepsilon_n^2 $, take $N(\varepsilon) = \exp(c_1 n \varepsilon_n^2) $ and $\varepsilon = MC^{-1} \varepsilon_n$, $j =1 $ in Lemma~\ref{prop-convex-test}, where $M > 4C$ is a large constant to be chosen later, there exist tests $\varphi_n$ with errors
    \begin{displaymath}
        P^n_0 \varphi_n \le e^{c_1 n \varepsilon_n^2} \frac{e^{-n M^2 C^{-2} \varepsilon_n^2 /32}}{1 - e^{-n M^2 C^{-2} \varepsilon_n^2 /32}}, \quad 
        \sup_{p\in \mathcal{P}_n: d(p,p_0) > M\varepsilon_n} P^n (1 - \varphi_n) \le e^{-n M^2 C^{-2} \varepsilon_n^2 /32}. 
    \end{displaymath}
Next the proof proceeds as in \cite{ghosal_convergence_2000}.
All terms should tend to zero for $ M^2 / (32 C^{2} ) > c_1 $ and $M^2 / (32 C^{2} ) > 2 + c_2 $. 
\end{proof}

\begin{lemma}\label{prop-convex-test}
Let $d$ be a discrepancy measure in the sense of (a)--(d) whose balls are convex and which is
  bounded from above by the Hellinger distance $h$.  If $N(C^{-1}\varepsilon/4 , \mathcal{Q}, d) \le
  N(\varepsilon)$ for any $\varepsilon>C\varepsilon_n >0$ and some non-increasing function $N:
  (0,\infty) \rightarrow (0,\infty)$, then for every $\varepsilon > C\varepsilon_n$ and $n$, there
  exists a test $\varphi_n$ such that for all $j\in \mathbb{N}$,
    \begin{displaymath}
        P^n \varphi_n \le N(\varepsilon) \frac{e^{-n\varepsilon^2/32}}{1-e^{-n\varepsilon^2/32}}, 
\quad \sup_{Q \in \mathcal{Q}, d(P,Q) > C j\varepsilon} Q^n (1-\varphi_n) \le e^{-n\varepsilon^2 j^2 /32}.
    \end{displaymath}
\end{lemma}

\begin{proof}
For a given $j \in \mathbb{N}$, choose a maximal set $Q_{j,1}, Q_{j,2},\dots, Q_{j, N_j}$ in the
  set $\mathcal{Q}_j = \{Q\in \mathcal{Q}: Cj \varepsilon < d(P,Q) < 2C j \varepsilon \}$ such that
  $d(Q_{j,k}, Q_{j,l}) \ge j\varepsilon/2$ for every $k \neq l$.  By property (d) of the discrepancy
every ball in a cover of  $\mathcal{Q}_j$ by balls of radius $C^{-1}j\varepsilon/4$ contains at most one $Q_{j,k}$. Thus
  $N_j \le N(C^{-1}j\varepsilon/4, \mathcal{Q}_j, d) \le N(\varepsilon)$.  Furthermore, the $N_j$ balls
  $ B_{j,l}$ of radius $j\varepsilon/2$ around $Q_{j,l}$ cover $\mathcal{Q}_j$, as otherwise the
  set of $Q_{j,l}$ would not be maximal.  For any point $Q$ in each $B_{j,l}$, we have
    \begin{displaymath}
        d(P, Q) \ge C^{-1} d(P,Q_{j,l}) - d(Q, Q_{j,l}) \ge j\varepsilon/2.
    \end{displaymath}
    Since the Hellinger distance bounds $d$ from above, also $h(P, B_{j,l}) \ge j\varepsilon/2$.  By
    Lemma~\ref{tests-convex-sets}, there exist a test $\varphi_{j,l}$ of $P$ versus $B_{j,l}$ with
    error probabilities bounded from above by $e^{-nj^2\varepsilon^2 /32}$.  Let $\varphi_n$ be the
    supremum of all the tests $\varphi_{j,l}$ obtained in this way, for $j = 1, 2, \dots$, and $l =
    1, 2, \dots, N_j$.  Then,
\begin{align*}
 P^n \varphi &\le \sum_{j=1}^{\infty}\sum_{l=1}{N_j} e^{-n j^2 \varepsilon^2 /32} \le \sum_{j=1}^{\infty} N(C^{-1} j \varepsilon/4, \mathcal{Q}_j, d) e^{-nj^2 \varepsilon^2/32} \\
&\le N(\varepsilon) \frac{e^{-n\varepsilon^2/32}}{1- e^{-n\varepsilon^2/32}},
\end{align*}
    and for every $j\in \mathbb{N}$,
    \begin{displaymath}
        \sup_{Q\in \cup_{l>j} \mathcal{Q}_l} Q^n (1 - \varphi_n) \le \sup_{l>j} e^{-nl^2 \varepsilon^2/32} \le e^{-nj^2 \varepsilon^2/32},
    \end{displaymath}
    by the construction of $\varphi_n$. 
\end{proof}

The following lemma comes from the general results of 
\cite{birge1984MR764150} and \cite{le_cam_asymptotic_1986}. 

\begin{lemma} 
\label{tests-convex-sets}
For any probability measure $P$ and dominated, convex set of probability measures $\mathcal{Q}$ with $h(p,q) > \varepsilon$ for any $q \in \mathcal{Q}$ and any $n\in \mathbb{N}$, there exists a test $\phi_n$ such that 
\begin{displaymath}
P^n \phi_n \le e^{-n\varepsilon^2 /8}, \quad \sup_{Q \in \mathcal{Q}} Q^n(1-\phi_n) \le e^{-n\varepsilon^2/8} 
\end{displaymath}
\end{lemma}

\bibliographystyle{abbrvnat}
\bibliography{nguyen.bib}
% \appendix
% \section{Appendix}
\end{document}